\documentclass{amsart}

\usepackage[mathscr]{eucal}
\usepackage{graphicx}
\usepackage{latexsym}
\usepackage{amssymb}
\usepackage{enumerate}
\DeclareMathAlphabet{\mathpzc}{OT1}{pzc}{m}{it}

\include{psbox}

\usepackage{amsfonts}
\usepackage{graphicx}
\usepackage{amsfonts}
\usepackage{color}
\usepackage[usenames,dvipsnames]{xcolor}
\usepackage[textsize=small]{todonotes}

\DeclareMathOperator*{\argmax}{arg\,max}

\usepackage[numbers]{natbib}

\newtheorem{theorem}{\bf Theorem}[section]

\newtheorem{corollary}{\bf Corollary}[section]

\newtheorem{lemma}{\bf Lemma}[section]

\newtheorem{condition}{\bf Condition}[section]
\newtheorem{proposition}{\bf Proposition}[section]

\numberwithin{equation}{section}
\newcommand{\skp}{\vspace{\baselineskip}}
\newcommand{\noi}{\noindent}

\theoremstyle{remark}

\newtheorem{algorithm}{\bf Algorithm}[section]

\newcommand{\eps}{\varepsilon}

\newcommand{\RR}{\mathbb{R}}
\newcommand{\ZZ}{\mathbb{Z}}

\newcommand{\V}{\mathcal{V}}
\newcommand{\E}{\mathbb{E}}

\newcommand{\PP}{\mathbb{P}}

\newcommand{\ti}{\tilde}

\newcommand{\MM}{\mathbb{M}}

\newcommand{\lf}{\lfloor}
\newcommand{\rf}{\rfloor}

\newcommand{\Om}{\mathnormal{\Omega}}
\newcommand{\om}{\omega}

\newcommand{\beq}{\begin{eqnarray*}}
\newcommand{\eeq}{\end{eqnarray*}}
\newcommand{\beqn}{\begin{eqnarray}}
\newcommand{\eeqn}{\end{eqnarray}}

\newcommand{\bt}{\begin{theorem}}
\newcommand{\et}{\end{theorem}}

\newcommand{\be}{\begin{equation}}
\newcommand{\ee}{\end{equation}}

\newcommand{\NN}{\mathbb{N}}

\newcommand{\clw}{\mathcal{W}}

\newcommand{\cll}{\mathcal{L}}

\newcommand{\clf}{\mathcal{F}}

\newcommand{\N}{\mathbb{N}}
\newcommand{\Q}{\mathbb{Q}}

\numberwithin{equation}{section} \numberwithin{lemma}{section}

\title[Hydrodynamic Source Detection Algorithms]{Source detection algorithms for dynamic contaminants based on the analysis of a hydrodynamic limit.}

\author{ Sergio A. Almada Monter, Amarjit Budhiraja and Jan Hannig}

\thanks{Amarjit Budhiraja's Research is supported in part by the National Science Foundation (DMS-1004418, DMS-1016441, DMS-1305120) and the Army Research
 Office (W911NF-10-1-0158, W911NF-14-1-0331). Jan Hannig's Research  supported in part by the National Science Foundation (DMS-1511945, DMS-1007543, DMS-1016441).}

\begin{document}

\begin{abstract}
In this work we propose and numerically analyze an algorithm for detection of a contaminant source using a dynamic sensor network.  The algorithm is motivated using a global probabilistic optimization problem and is based on the analysis of the hydrodynamic limit of a discrete time evolution equation on the lattice under a suitable scaling of time and space.  Numerical results illustrating the effectiveness of the algorithm are presented. \\

\noi {\bf AMS 2000 subject classifications:} 60G35, 65C35, 86A22, 93E10. \\

\noi {\bf Keywords:} source detection algorithms,  inverse problems, data assimilation, particle systems, hydrodynamic limits.
\end{abstract}

{%

\def\N{\mathbb{N}}
\def\P{\mathbb{P}}
\def\F{\mathcal{F}}
\def\R{\mathbb{R}}
 \def\X{\mathbb{X}}
 \def\Y{\mathbb{Y}}
\def\S{\mathbb{S}}
\def\W{\mathbb{W}}
\def\V{\mathbb{V}}
\def\U{\mathbb{U}}
\def\C{\mathbf{C}}
\def\B{\mathcal{B}}
\def\Lexp{\mathcal{L}_{\mbox{\tiny{exp}}}}
\def\strategy{h}
\def\Prob{\mathbb{P}}	
\def\Q{\mathbb{Q}}		
\def\Eproc{\mathbb{G}}
\def\E{\mathbb{E}}			
\def\Eb{\overline{\mathbb{E}}}
\def\calR{\mathcal{R}}
\def\frakF{\mathfrak{F}}
\def\linf{\ell _{\infty}(\mathfrak{F})}
\def\lin{\ell _{\infty}}
\newcommand{\edfis}{\tilde{\mathbf{F}}}
\newcommand{\empd}{\mathbf{F}}
\newcommand{\Rinf}{\mathbb{R}^{\infty}}	
\newcommand{\M}{\mathcal{M}}	
\newcommand{\bM}{\mathbb{M}}	
\newcommand{\re}{\mathcal{H}}	
\newcommand{\G}{\mathcal{G}}	
\newcommand{\A}{\mathcal{A}}	
\newcommand{\Var}{\textrm{Var}}
\newcommand{\Exp}{\mathbb{E}}
\newcommand{\calF}{\mathcal{F}}	
\newcommand{\calE}{\mathcal{E}}	
\newcommand{\lebT}{\lambda _{T}}	
\newcommand{\lebinf}{\lambda _{\infty}}	
\newcommand{\cadlag}{c\`{a}dl\`{a}g }

\maketitle  
\section{Introduction.}
Source detection of a diffusing chemical/biological contaminant based on limited sensor measurements is a challenging inverse problem that arises in many environmental, geophysical and defense applications (see \cite{HHCGPR} and references therein).
In contrast to partial differential equation(PDE) plume models considered in much of the literature\cite{ChMoSo, ishida2001plume, rao2005identification, li2001tracking, sykes1986gaussian}, here our starting point is an elementary particle model in a planar domain. One of the reasons for considering such a microscopic particle model is that we are interested in a certain optimization problem that can be analyzed exactly using the particle approach. In PDE based models a problem formulation that is commonly considered in literature \cite{christopoulos2005adaptive, brennan2004radiation, nemzek2004distributed} is one where the parameters governing the PDE need to be estimated based on measurements obtained from a given network of sensors.  However, advances in sensor technologies make it possible to consider a {\em dynamic} sensor network that not only solves the source detection problem but also aims for optimality under appropriate performance criteria.  Development of such detection algorithms is the goal of this work.

In order to highlight our main approach to dynamic sensor assignment algorithms we focus on a simple probabilistic model for contaminant evolution.  This model can be generalized to include many application specific features of transport, advection, diffusion, adsorption, decay, temporal and spatial inhomogeneities etc. The basic model is as follows.  Consider a pollutant described through a collection of particles  
%
%
%
%
on the  two dimensional lattice $\ZZ^2_h = h\ZZ^2$, where $\ZZ$ is the set of integers and $h>0$ is a small parameter.
At each discrete time instant $0, h, 2h, \ldots$, a random number of particles are introduced at a source site $e \in \ZZ^2_h$.  Subsequent to its entrance into the system, each particle
undergoes a nearest neighbor spatially homogeneous random walk.  We are given a sensor system that at any given time instant can measure the number of particles at up to $r^2$ distinct sites.
We assume that the statistical parameters that govern the random walks of the particles are known or can be reliably estimated by independent means.  The goal of the observer is to find the source site in as few time steps as possible.

In this work we propose and numerically analyze an algorithm that is motivated by a global optimization problem.
The objective function $\psi: \ZZ^2_h \to \R_+$ for this optimization problem takes the form $\psi(x) = \E_x \ell_{\infty}(X)$ where
$X = \{X_n\}_{n \in \N_0}$ represents the Markov chain associated with a typical particle and $\E_x$ denotes the expected value w.r.t. the probability measure under which $X_0=x$ a.s.
The function $\bar\ell_{\infty}: (\ZZ^2_h)^{\infty} \to \R_+$ that describes the objective function will be introduced in Section \ref{sec:optform}.
We show that under Condition \ref{cond:transient} $\psi$ achieves its maximum uniquely at the point $e$.  In view of this result the problem of source detection
can be translated to finding the point where $\psi$ is maximized. 
Next, we introduce a natural estimator for $\psi(w)$ for each $w \in \ZZ_h^2$.  This estimator can be computed by deploying a sensor at $w$ that measures the number of particles at the site
for $n$ units of time.  We show that the estimator is consistent as $n \to \infty$ thus if  an infinite number of sensors were available, one could determine 
an approximate solution of the optimization problem by computing $\argmax_{w \in \ZZ_h^2} \psi_n(w)$.  Since, however, only $r^2$ sensors are available at any given time instant, an efficient scheme for placement of the sensors becomes critical.  In Section \ref{sec:prelim} we introduce our first algorithm for dynamic sensor placement that has the advantage of not requiring the knowledge of underlying statistical distribution of the random walks, 
however in general it requires a large number of time steps to converge. Section \ref{sec: algorithm2} presents an alternative  approach that is based on an analysis of the hydrodynamic limit, under a suitable spatial and temporal scaling, of a 
certain approximation of the estimator $\psi_n(w)$.  The limit is governed by a transport equation given in \eqref{eq:eq522} the solution of which can be explicitly given.  Using the form
of this solution we develop a search algorithm that dynamically allocates $r^2$ sensors in suitable regions of the state space until the source site is discovered.  Sensor placement sites are determined by the paths of certain Brownian motions with a drift where the drift vector is determined from the statistical parameters of the underlying Markov chain while the variance of the Brownian motion is a tuning parameter which is modulated dynamically in response to the output of the search algorithm.

In Section~\ref{sec: description} we give the description of the problem and a formulation in terms of an optimization problem. Section~\ref{sec:estimator1} gives a weakly consistent estimator for $\psi(x)$ for each $x \in \ZZ_h^2$
and Section \ref{sec:prelim} introduces a preliminary algorithm that is motivated by the estimator constructed in Section \ref{sec:estimator1}. In Section \ref{sec: hydrodynamic} we analyze a hydrodynamic limit under a suitable
spatial and temporal scaling.  The form of the hydrodynamic limit suggests a natural refinement of the algorithm in Section \ref{sec:prelim} which is introduced in Section \ref{sec: algorithm2}. Finally in Section  \ref{sec: numerical}
we present results of some numerical experiments.

\section{Problem Description and Foundations}\label{sec: description}
Let $h>0$ and consider the following mechanism of spread of pollutant particles originating from an unknown source site $e$ on the scaled integer lattice
$h\ZZ^2$.
\begin{itemize}
	\item The state of the system changes at discrete time instants $kh$, for $k = 0, 1, 2, \ldots$.
	\item At each time instant $nh$, $n\ge 0$, a random number $\alpha_n$ of particles are injected at the unknown
	source site $e \in h\ZZ^2$.  We assume that $\{\alpha_n\}_{n\ge 0}$ is an i.i.d. sequence and that $\E \alpha_n = \alpha h \in (0,\infty)$.
	\item Each particle, subsequent to its entrance in the system, independently of all other particles, follows a nearest neighbor 
	random walk on $h\ZZ^2$ in discrete time, with time indexed as $\{jh\}_{j \ge j_0}$, where $j_0h$ represents the time at which the particle enters the system.
	\item Random walks of all the particles have the same transition kernel which is described in terms of positive scalars $p_l$, $l=1, \ldots , 4$, satisfying $\sum_{l=1}^4 p_l=1$.  Specifically, if for $j \ge 1$, $X_j$ denotes the random location of a particle at time instant $nh$ that enters the system at time $0$, then
	\begin{equation}\label{eq:eq1125}
		\PP[X_n = x_l \mid X_{n-1} =x] = p_l,\; l=1, \ldots , 4
	\end{equation}
	where for $x \in h\ZZ^2$, $x_l = x+ he_l$ and $e_1=(1,0)'$, $e_2=(0,1)'$, $e_3= -e_1$
	and $e_4= -e_2$.
\end{itemize}
We further assume that there is an observer who knows the vector $ \mathbf{p} = (p_1, p_2, p_3, p_4)$, the distribution of $\alpha_1$ and 
has the resources to measure the number of particles at any given time instant simultaneously at up to $r^2$ lattice sites.
The goal of the observer is to identify the source site $e$ in as few steps as possible and with the minimum amount of effort.

\subsection{An Optimization Formulation} \label{sec: optimization}
\label{sec:optform}
For $v \in h\ZZ^2 \doteq \ZZ_h^2$, let $\PP_{v}\equiv \PP_v^h$ denote the unique probability measure on
$(\ZZ_h^2)^{\infty}$, under which the canonical sequence 
$$X_n(\om) = \om_n, \; \om = (\om_1, \om_2, \ldots)\in (\ZZ_h^2)^{\infty}$$
is a Markov chain with transition probabilities as in \eqref{eq:eq1125} and such that $X_0(\om) = v$, $\PP_v$ a.s.
For $v \in \ZZ_h^2$, let
$$\ell_{\infty}(v) = \sum_{n=0}^{\infty}\delta_{v}(X_n),$$ where for $v_1, v_2 \in \ZZ_h^2$,
$\delta_{v_1}(v_2)$ equals $1$ if $v_1=v_2$ and is $0$ otherwise.
We will make the following transience assumption on the Markov chain.
\begin{condition}
	\label{cond:transient}
	$|p_1-p_3| + |p_2-p_4| \neq 0$.
\end{condition}	

	Under the above condition we show in Lemma~\ref{lemma:transient} below that
	\begin{equation}
		\sum_{n \in \NN_0} n \PP_v(X_n = w) < \infty, \mbox{ for all } v,w \in \ZZ_h^2
		\label{eq:eq614}
	\end{equation}
	and consequently
	\begin{equation}\label{eq:1208}
		\sum_{n\in \NN_0} \PP_v(X_n \in K) < \infty \mbox{ for every compact } K \subset \RR^2 \mbox{ and } v \in \ZZ_h^2.
	\end{equation}
		\begin{lemma}
	\label{lemma:transient} Under Condition~\ref{cond:transient}, for every $v,w \in \ZZ_h^2$, there is a $c>0$ such that $$
	\sup_{ n>0} e^{cn} \PP_v( X_n = w ) < \infty.
	$$
	In particular~\eqref{eq:eq614} is satisfied.
	\end{lemma}
	\begin{proof}.
	Under Condition~\ref{cond:transient}  either $\E_v (X_1- v)\cdot e_1 \neq 0$ or  $\E_v ( X_1 - v )\cdot e_2 \neq 0$. Suppose without loss of generality that $ \mu_1 = \E_v (X_1 - v )\cdot e_1  \ne 0 $ . Then, upon defining $A_n = \left \{  \left \vert ( X_n - v )\cdot e_1  - \mu_1 n  \right \vert \leq n | \mu_1 |/2\right\}$, it is straightforward to see that 
	\begin{align*}
	\PP_v( X_n = w ) &\leq \PP_v( X_n \cdot e_1= w\cdot e_1 ) \\
	&\leq  \PP_v( X_n \cdot e_1= w\cdot e_1, A_n ) + ( 1 - \PP_v( A_n )) \\
	&\leq {\bf 1} _ { [n( \mu_1 - |\mu_1|/2), n( \mu_1 + |\mu_1|/2) ] } ( (w-v)\cdot e_1 )  \\ 
	& \qquad + \PP_v \left( \left \vert X_n\cdot e_1 - \E_v X_n \cdot e_1 \right | > n |\mu_1|/2  \right ).
	\end{align*}
	The result follows because Hoeffding's inequality~\cite{Hoeffding} implies that the last probability is bounded by $e^{ -\mu_1^2n/2}$, and since $ {\bf 1} _ { [n( \mu_1 - |\mu_1|/2), n( \mu_1 + |\mu_1|/2) ] } ( (w-v)\cdot e_1 ) = 0$, for all $n$ large enough.
	\end{proof}
	
	Condition \ref{cond:transient} will be assumed throughout this work and will not be explicitly mentioned in the statement of results.  
	
	The following result shows that $e$ can be characterized as the unique maximizer of the function $v \mapsto \E_{e}\ell_{\infty}(v)$.
	\begin{proposition}
		\label{prob:uniqmax}
	 Given $e \in \ZZ^2_h$, the function $w \mapsto \E_{e}\ell_{\infty}(w)$ from $\ZZ^2_h \to \RR_+$ attains its maximum uniquely at $w=e$.
			\end{proposition}
	\begin{proof}
	For $w \in \ZZ^2_h$, let $\tau_{w} = \inf\{n \in \N_0: X_n = w\}$.
	Note that for any $w \in \ZZ^2_h$
	$$\E_e\ell_{\infty}(w) =\PP_e(\tau_{w}<\infty) \E_{w}\ell_{\infty}(w) = \PP_e(\tau_{w}<\infty)\E_{e}\ell_{\infty}(e),$$
	where the first equality uses the strong Markov property and the second is a consequence of the spatial homogeneity of the
	transition probabilities.  In order to complete the proof it suffices to show that
	\begin{equation}
		\label{eq:eq1217}
		\PP_e(\tau_{w}<\infty) < 1 \mbox{ for all } w \in \ZZ^2_h\setminus \{e\}.
			\end{equation}
			From Condition \ref{cond:transient} one of the following four cases must hold: $p_1 > p_3$,
	$p_3 > p_1$, $p_2 > p_4$, or $p_4 > p_2$.  Without loss of generality we assume that $p_1>p_3$.
	Fix $w \in \ZZ^2_h\setminus \{e\}$ and let
	$$\tau^1_{w} = \inf\{n \in \N_0: X_n \cdot e_1 > w\cdot e_1\}.$$
	From the strict positivity of $\{p_l\}$ we have that
	$$\P_e(\tau^1_w < \tau_w) > 0.$$
	Since $p_1 > p_3$
	\begin{equation}
		\inf_{x \in \ZZ^2_h} \P_x( X_n\cdot e_1 \ge x \cdot e_1 \mbox{ for all } n \in \N_0)  \doteq \theta_1 > 0.
		\label{eq:eq342}
	\end{equation}
	Thus,
	$$
	\P_e(\tau_{w} = \infty) \ge \P_{e}(\tau_{w}^1 < \tau_w; X_n\cdot e_1 \ge X_{\tau_w^1}\cdot e_1 \mbox{ for all } n \ge 
	\tau^1_{w}) \ge \P_{e}(\tau_{w}^1 < \tau_w)\theta_1 > 0,$$
	where the second inequality follows from the strong Markov property and \eqref{eq:eq342}.  This proves
	\eqref{eq:eq1217} and the result follows. 
	\end{proof}
	
	In view of the above result, the source detection problem reduces to finding the maximum for the function
	$w \mapsto \E_e(\ell_{\infty}(w))$.  

\subsection{A Consistent Estimator} \label{sec:estimator1}
	Since the source $e$ is unknown to the observer, the quantity $\E_e(\ell_{\infty}(w))$ is not computable, however the
	following result gives a readily computable estimator for this quantity.
	
	For $n \in \N$ and $i = 1, 2, \ldots$, denote by $X^{(n,i)}= (X^{(n,i)}_j)_{j\in \N_0}$ the sequence
	of random variables that describes the random walk of the $i$-th particle introduced at the time instant $nh$.
	Although at any instant only a finite random number of particles are introduced, for notational convenience
	we work with a doubly infinite array of random variables.
	
	We denote the probability space that supports the random variables $\{\alpha_n, n \in \N\}$ and
	$\{(X^{(n,i)}_j)_{j \in \N_0}, n \in \N, i = 1, \ldots \alpha_n\}$ by $(\Om,\clf, \P)$ and the corresponding
	expectation by $\E$.  
	
	For $w \in \ZZ^2_h$ and $n \in \N$, let
	$N_n(w)$ be the number of particles at site $w$ at time instant $nh$.  This random variable can be measured by the observer
	by placing a sensor at site $w$ at time $nh$.  Note that
	$$N_n(w) = \sum_{j=0}^n \sum_{i=1}^{\alpha_j} \delta_w(X_{n-j}^{(j,i)}), \; n \in \N_0.$$
	Let $\lambda_n(w)$ be the average number of particles per unit time at the site $w$, namely
	\begin{align} \notag
		\lambda_n(w) &= \frac{1}{n} \sum_{m=0}^{n-1} N_m(w) \\ \notag
		&= \frac{1}{n} \sum_{m=0}^{n-1} \sum_{j=0}^m \sum_{i=1}^{\alpha_j} \delta_{w}(X_{m-j}^{(j,i)}) \\
		&= \frac{1}{n} \sum_{j=0}^{n-1} \sum_{m=j}^{n-1} \sum_{i=1}^{\alpha_j} \delta_{w}(X_{m-j}^{(j,i)}).
		\label{eq:eq439}
	\end{align}
The following result shows that $\frac{1}{h\alpha}\lambda_n(w)$ is a consistent estimator for $\E_e(\ell_{\infty}(w))$.
\begin{proposition}
	\label{prop:consest}
	For every $w \in \ZZ^2_h$, $\lambda_n(w) \to h\alpha \E_e(\ell_{\infty}(w))$ in $L^1(\P)$ as $n\to\infty$.
	\end{proposition}	
	\begin{proof} Fix $w \in \ZZ^2_h$, and define for $j \in \N$ and $n > j$
	$$\ell^{(j)}_{n-j}(w) = \sum_{m=j}^{n-1}\sum_{i=1}^{\alpha_j} \delta_w(X_{m-j}^{(j,i)}).$$
	Note that this quantity represents the total amount of time particles injected at time instant $jh$ spend at 
	site $w$ by time $(n-1)h$. From \eqref{eq:eq439}
	$$\lambda_n(w) = \frac{1}{n} \sum_{j=0}^{n-1} \ell_{n-j}^{(j)}(w), \; n \in \N.$$
	Define the sequence  $\{\ell_{\infty}^{(j)}(w)\}_{j\ge 1}$, given by
	$$\ell_{\infty}^{(j)}(w) = \sum_{m=j}^{\infty} \sum_{i=1}^{\alpha_j} \delta_w(X_{m-j}^{(j,i)}), \quad j \in \N.$$
	Clearly $\{\ell_{\infty}^{(j)}(w)\}_{j\ge 1}$ is an i.i.d. sequence so that each variable has the same mean given as
$$
\E\ell^{(j)}_{\infty}(w) = \E(\alpha_j) \E\sum_{m=j}^{\infty}\delta_w(X_{m-j}^{(j,1)}) = h\alpha \E_e\ell_{\infty}(w)
= h\alpha \sum_{k=0}^{\infty}\P_e(X_k=w) < \infty.$$
For the first equality we used Wald's lemma, while the fact that the sum is finite is a consequence of~\eqref{eq:eq614}.
Note that
\begin{equation}
\label{eq:514}
\lambda_n(w) = \frac{1}{n}\sum_{j=0}^{n-1}\ell^{(j)}_{\infty} - \frac{1}{n}\sum_{j=0}^{n-1} L_n^{(j)}(w)	
\end{equation}
where $L_n^{(j)} = \ell^{(j)}_{\infty} - \ell^{(j)}_{n-j}$.  We now argue that the second term on the right side
of the above display converges to $0$  in $L^1(\P)$.  Note that, for $j<n$,
\begin{align*}
\E \sum_{m=n}^{\infty} \sum_{i=1}^{\alpha_j} \delta_w(X_{m-j}^{(j,i)})	
= h\alpha \sum_{m=n}^{\infty} \P(X_{m-j}^{(j,1)}=w) = h\alpha \sum_{m=n-j}^{\infty} \P_e(X_{m}=w).
\end{align*}
Therefore, 
\begin{align*}
	\frac{1}{n}\sum_{j=0}^{n-1} \E L_n^{(j)}(w)	&= \frac{h\alpha}{n}\sum_{j=0}^{n-1} \sum_{m=n-j}^{\infty} \P_e(X_{m}=w)\\
	&= \frac{h\alpha}{n}\sum_{j=1}^{n} \sum_{m=j}^{\infty} \P_e(X_{m}=w)\\
	& = 
	\frac{h\alpha}{n}\sum_{m=1}^{\infty} m \P_e(X_{m}=w).
\end{align*}
In view of \eqref{eq:eq614} the last expression approaches $0$ as $n\to \infty$. 
Finally from the law of large numbers $\frac{1}{n}\sum_{j=0}^{n-1}\ell^{(j)}_{\infty}$	converges
in $L^1(P)$ to $h\alpha \E_e(\ell_{\infty}(w))$.  The result now follows by combining the above observations.
\end{proof}

We note the following immediate corollary of the proposition.
\begin{corollary}
	\label{cor:randcons}
	Let $\tau$ be a $\N_0$ valued random variable on $(\Om, \clf, \P)$.  Then
	$\frac{1}{n} \sum_{m=\tau+1}^{\tau+n} N_m \to h\alpha \E_e\ell_{\infty}(w)$ in probability.
\end{corollary}
\begin{proof} Write
\begin{equation}
	\label{eq:eq1249}
	\frac{1}{n} \sum_{m=\tau+1}^{\tau+n} N_m(w) = \frac{1}{n} \sum_{m=0}^{n-1} N_m(w) - 
	\frac{1}{n} \sum_{m=0}^{\tau} N_m(w) + \frac{1}{n} \sum_{m=n}^{\tau+n} N_m(w).
\end{equation}
Since $\tau < \infty$ and $N_m<\infty$ a.s. for all $m \in \N_0$,
\begin{equation}
	\label{eq:eq1252}
	\frac{1}{n} \sum_{m=0}^{\tau} N_m(w) \to 0 \mbox{ a.s. as } n \to \infty .
\end{equation}
For a compact $K \subset \ZZ_h^2$, let $N_n(K)$ denote the number of particles in $K$ at time instant $nh$.
Then
$$\E(N_n(K)) = \E\sum_{j=0}^n \sum_{i=1}^{\alpha_j} \delta_K(X_{n-j}^{(j,i)}) = h\alpha \sum_{j=0}^{n} \P_e(X_j \in K),$$
where $\delta_K(v) = 1$ if $v \in K$ and $0$ otherwise.  Combining the above display with \eqref{eq:1208}
we now have that
\begin{equation}
	\label{eq:eq103}
	\frac{N_n(K)}{n} \to 0 \mbox{ in probability , as } n \to \infty, \mbox{ for every compact } K \subset \ZZ_h^2. 
\end{equation}
Next note that, for each $r \in \N_0$ and $\eps > 0$
\begin{equation}
	\label{eq:eq106}
	\P(\frac{1}{n} \sum_{m=n}^{\tau+n} N_m(w) > \eps) \le \P(\frac{1}{n} \sum_{m=n}^{\tau+n} N_m(w) > \eps; \tau \le r)
	+ \P(\tau > r).
\end{equation}
Let $K_r = \{z \in \ZZ_h^2: \mbox{dist}(z,w) \le r\}$ where $\mbox{dist}$ denotes the usual graph
distance on $\ZZ_h^2$.
Since the particles follow a nearest neighbor random walk, as $n \to \infty$,
$$
\P(\frac{1}{n} \sum_{m=n}^{\tau+n} N_m(w) > \eps; \tau \le r) \le \P(\frac{1}{n}N_n(K_r) > \eps) \to 0.$$
Also since $\tau<\infty$ a.s., $\P(\tau >r) \to 0$ as $r \to \infty$.  Using these two observations in 
\eqref{eq:eq106} we now have that
$\frac{1}{n} \sum_{m=n}^{\tau+n} N_m(w)$ converges to $0$ in probability as $n \to \infty$.
The result now follows on combining this with \eqref{eq:eq1252}, \eqref{eq:eq1249} and Proposition \ref{prop:consest}.
\end{proof}

\subsection{A Preliminary Algorithm}
\label{sec:prelim}
Proposition \ref{prop:consest} and Corollary \ref{cor:randcons}	suggest the following natural approach to the discovery of the source site.
\begin{algorithm}
	\label{alg:algstraw}
	Suppose at some time instant $n_1$ a sensor has detected positive number of particles at site $w^{(1)}$.
	Fix $r>0$ and $N_0 \in \N_0$.  These numbers represent the parameters for the scanning window and time window, respectively, introduced below.
	\begin{enumerate}
		\item Compute the average number of particles in the time window $[n_1+1, n_1+N_0]$ at site $w$, i.e.
		\begin{equation}
			\lambda_{n_1}^{N_0}(w) \doteq \frac{1}{N_0} \sum_{m=n_1+1}^{n_1+N_0} N_m(w),\label{eq:eq958}
		\end{equation}
		for all $w$ in the scanning window
		$$\clw_r(w^{(1)}) = \{w \in \ZZ_h^2: |w^{(1)}_k - w_k| \le rh/2, \; k = 1,2\}.$$
		Let $w^* = \argmax_{w\in \clw_r(w^{(1)})}  \lambda_{n_1}^{N_0}(w)$.  Set $w^{(2)} = w^*$ and $n_2 = n_1+N_0$.
		\item Define recursively sequences $(n_i, w^{(i)})_{i\ge 1}$ as follows.  Having defined,
		$(n_i, w^{(i)})_{i=1}^j$, set $n_{j+1} = n_j+N_0$ and $w^{(j+1)}= w^*$ where $w^*$ is obtained by following step 1
		with $(n_j, w^{(j)})$ replacing $(n_1, w^{(1)})$.
		\item Stop when the sequence $w^{(i)}$ converges.
		\end{enumerate}
\end{algorithm}
Convergence in the above algorithm is defined by specifying a suitable tolerance threshold.  Note that the algorithm requires using $r^2$ sensors at each time step.  
We also remark that the algorithm does not require the knowledge of the parameters $\alpha$ or the probability vector $\mathbf{p}$. Finally note that, we have not specified how $w^{(1)}$ is determined.
This depends on the problem setting.  If initially no information is available, one may do a random search using
a fixed number of sensors at each time instant until a site with particles is discovered.  Typically, there will be
some initial information available and the search algorithm should appropriately incorporate this information in selecting the initial site.  
This issue will not be addressed in the current work.  Section \ref{sec: numerical} presents some numerical results on the implementation of this algorithm.

\section{Hydrodynamic Scaling} \label{sec: hydrodynamic}
In Section \ref{sec: algorithm2} below we will propose an alternative scheme that makes a more careful use 
of the underlying statistical law of the particles and as a consequence uses a more effective placement of sensors at any given
time instant.  We begin with the 
observation that as $n\to \infty$
$$
\mu_n(w) = \E N_n(w) = \alpha h \sum_{k=0}^n \P_e(X_k=w) \to \alpha h \E_e\ell_{\infty}(w).$$
Thus for large $n$, $\mu_n(w)$ is a good approximation for $\alpha h \E_e\ell_{\infty}(w)$ and maximizer of $\mu_n(w)$
gives an approximate solution of the optimization problem.
We will now describe an evolution equation for
$\{\mu_n(w), n \in \N_0, w \in \ZZ_h^2\}$ and consider a scaling limit of this equation which leads to a more efficient
search scheme.
For $w \in \ZZ_h^2$, let
$w^{(i)} = w+he_i$, $i = 1, \ldots, 4$, where 
$\{e_i\}$ are as  introduced below \eqref{eq:eq1125}.
Also, for $w', w \in \ZZ_h^2$ and $m = 1, \ldots, N_{n-1}(w')$, let
$$
u_n^m(w',w) = \left \{ 
\begin{array}{cc}
	1 & \text{ if the $m$-th particle at $w'$ at time instant $h(n-1)$ moves to $w$ at time $hn$ }   \\ 
 0& \text{ otherwise. }  \\
\end{array} \right.
$$
Then, for $n\ge 1$,
$$N_n(w) = \sum_{l=1}^4 \sum_{k=1}^{N_{n-1}(w^{(l)})} u_n^k(w^{(l)},w) + \alpha_n \delta_{e}(w).$$
Taking expectations we get
$$
\mu_n(w) = \sum_{l=1}^4 \mu_{n-1}(w^{(l)}) \tilde p_l + h \alpha \delta_{e}(w),$$
where $\tilde p_1 = p_3$, $\tilde p_2 = p_4$, $\tilde p_3 = p_1$ and $\tilde p_4 = p_2$.  We can rewrite
the above equation as
\begin{equation}
	\label{eq:eq433}
	\mu_{n+1}(w) - \mu_n(w) = \sum_{l=1}^4 (\mu_n(w^{(l)}) - \mu_n(w)) \tilde p_l + \alpha h \delta_e(w).
\end{equation}
Define $\tilde \mu^h: [0, \infty)\times \R^2 \to \R_+$ as
$\tilde \mu^h(t,x) = \mu_{\lf t/h \rf} (h\lf x/h\rf)$.
Also, for $f:\R_+ \to \R$ and $t>0$, define
$$
\partial^hf(t)=\frac{1}{h} \left( f(h\lf t/h\rf +h) - f(h\lf t/h\rf)\right)$$
and for $g:\R^2 \to \R$, $w\in \ZZ_h^2$, define
\begin{align*}
	\nabla_1^{h,+}g(x) &= \frac{1}{h} \left( g(x_1^h) - g(x^h)\right), \; \nabla_1^{h,-}g(x) = \frac{1}{h} \left( g(x^h) - g(x^h_3)\right)\\
	\nabla_2^{h,+}g(x) &= \frac{1}{h} \left( g(x_2^h) - g(x^h)\right), \; \nabla_2^{h,-}g(x) = \frac{1}{h} \left( g(x^h) - g(x^h_4)\right),
\end{align*}
where $x = h\lf x/h\rf$, $x_i^h = h\lf x/h\rf + e_i h$, $i=1,\ldots 4$.
Evaluating \eqref{eq:eq433} with $w = h\lf x/h\rf$, $n=\lf t/h\rf$ and dividing by $h$ throughout, we have using the above notation, for all $(t,x) \in \R_+\times \R^2$
$$
\partial^h \tilde \mu^h(t,x)
= p_3 \nabla_1^{h,+} \tilde \mu^h(t,x) -  p_1 \nabla_1^{h,-} \tilde \mu^h(t,x)
+ p_4 \nabla_2^{h,+} \tilde \mu^h(t,x) - p_2 \nabla_2^{h,-} \tilde \mu^h(t,x) + \alpha \delta_e(h\lf x/h\rf).
$$
Formally taking limit as $h\to 0$, we are led to the following PDE
\begin{align}\label{eq:eq522}
	\frac{\partial \mu(t,x)}{\partial t} &= - q\cdot \nabla \mu(t,x) + \alpha \delta_{e}(t,x), \; (t,x) \in \R_+\times \R^2\nonumber\\
	\mu(0,x) &= 0, \; x \in \R^2,
\end{align}
where $q = (p_1-p_3, p_2-p_4)'$.

In Theorem \ref{thm:cgce} below we will make the above convergence mathematically precise.  We begin with some notation. Let $\MM$ denote the space of finite measures on $\R^2$
equipped with the topology of weak convergence.  This topology can be metrized in a manner that $\MM$ is a Polish space (cf.\cite{Billingsley}).
Let $C([0, \infty):\MM)$ denote the space of continuous functions from $[0,\infty)$ to $\MM$ equipped with the local uniform topology and let $C_0^1((0,\infty)\times \R^2)$ be the space of real valued continuously differentiable functions 
on $(0,\infty)\times \R^2$ with compact support.  The following result gives the wellposedness of the equation in \eqref{eq:eq522}. 
\begin{theorem}
	\label{thm:uniqweak}
	Equation \eqref{eq:eq522} has a unique weak solution in $C([0, \infty):\MM)$ given as
	\begin{equation}\label{eq:eq529}
		\mu(t,dx) = \alpha \int_0^t \delta_{e+qs}(dx) ds, \; t \ge 0,
	\end{equation}
	namely,  $\mu \equiv  \{\mu(t,dx)\}_{t\ge 0}$ given by \eqref{eq:eq529} is the unique element in $C([0, \infty):\MM)$ 
	such that $\mu(0,dx)=0$ and for $t> 0$ and all $f \in C_0^1((0,\infty)\times \R^2)$
	\begin{equation}
		\label{eq:eq532}
		\int_{(0,\infty)\times \R^2} \frac{\partial f}{\partial t}(t,x) \mu(t,dx) dt +
		\int_{(0,\infty)\times \R^2} q \cdot \nabla f(t,x) \mu(t,dx) dt + \alpha \int_{(0,\infty)} f(t,e) dt = 0.
		\end{equation}
\end{theorem}
\begin{proof}
	The fact that $\mu$ defined in \eqref{eq:eq529} is a weak solution of \eqref{eq:eq522} is an immediate
	consequence of the identity
	$$\int_{0}^{\infty} \left(\frac{d}{dt} \int_0^t f(t, e+ q(t-s)) ds \right) dt = 0$$
	for every $f \in C_0^1((0,\infty)\times \R^2)$.
For uniqueness note that if $\mu_1$, $\mu_2$ are two weak solutions  of \eqref{eq:eq522} then
the signed measure $\mu = \mu_1-\mu_2$ solves the homogeneous equation
\begin{align*}
	\frac{\partial \bar\mu(t,x)}{\partial t} &= - q\cdot \nabla \bar \mu(t,x), \; (t,x) \in \R_+\times \R^2\nonumber\\
	\bar\mu(0,x) &= 0, \; x \in \R^2,
\end{align*}
whose unique solution is $\bar \mu(t,dx) = 0$.
\end{proof}

For $t\ge 0$, define $\mu^h_t \in \MM$ as
$\mu^h_t(A) = \sum_{x\in A \cap \ZZ_h^2} \tilde \mu^h(t,x)$.  Let $D([0, \infty):\MM)$ denote the space of  functions from $[0,\infty)$ to $\MM$ that are right continuous and have left limits, equipped with the usual Skorohod topology.
Note that $\mu^h = (\mu^h_t)_{t\ge0}$ is an element of $D([0, \infty):\MM)$.  The following result gives the convergence
of $\mu^h$ to $\mu$ as $h \to 0$.
\begin{theorem}
	\label{thm:cgce}
	As $h \to 0$, $\mu^h\to \mu$ in $D([0, \infty):\MM)$.
\end{theorem}
\begin{proof}
For $f \in C_b(\R^2)$
\begin{align}
	\int_{\R^2} f(x) \mu^h_t(dx) &= \sum_{w \in \ZZ_h^2} \mu_{\lf t/h\rf}(w) f(w)
	= \alpha h \sum_{w \in \ZZ_h^2} \sum_{m=1}^{\lf t/h\rf}\P_e(X_m=w) f(w)\nonumber\\
	&= \alpha h \sum_{m=1}^{\lf t/h\rf} \E_e f(X_m) = 
	\alpha \int_0^t \E_e f(X_{\lf s/h\rf}) ds + O(h)\\
	&= \alpha \int_0^t \E_e f(X^h(s)) dA^h(s) + O(h),
	\label{eq:eq852}
\end{align}
where $A^h(s) = jh$, $s \in [jh, (j+1)h)$, $j = 0, 1, \ldots$ and
$$
X^h(s) = \frac{s-jh}{h}X_j + \frac{(j+1)h-s}{h} X_{j+1}, \;  s\ge 0.$$
Note that, under $\PP_e$, $X^h \Rightarrow x$ in $C([0,\infty):\R^2)$, where
$x(t) = e+q t$, $t\ge 0$.
Also, $A^h \to A$ in $D([0,\infty):\R_+)$ where $A(t)=t$, $t\ge 0$.
Thus the right side of \eqref{eq:eq852} converges to 
$$\alpha \int_0^t f(x(s)) ds = \int_{\R^2} f(x) \mu_t(dx)$$
in $D([0,\infty):\R)$ for
every $f \in C_b(\R^2)$.  The result follows. 
\end{proof}

\subsection{An Algorithm Based on the Hydrodynamic Limit Analysis} \label{sec: algorithm2}
Solution of the PDE in \eqref{eq:eq522} says that for $h$ small, $\mu_n$ evolves as follows. Initially, $\mu_0=0$
and for $n>0$, $\mu_n(w) \approx \alpha$ if $w$ is in  a small neighborhood of the set 
$\cll_n = \{e+tq: t \le nh\}$
This suggests a natural form of search algorithm which is based on the following heuristic:
If a sensor has detected a large number of particles at a site $w_0 \in \ZZ_h^2$ then it must be close to the line
$\cll = \{e+tq: t \ge 0\}$.  Thus by exploring sites in the neighborhood of $\{w_0-qt, t \ge 0\}$ one should be able to
efficiently find sites with high value of $N_n(w)$ and eventually discover the source site $e$. In order to account
for the randomness in $N_n(w)$ we replace the trajectory $\hat x(t) = w_0-qt$ with the stochastic process
$$\hat X_h(t) = w_0-qt + \sqrt{h}\Lambda W(t),$$
where $W$ is a standard two dimensional Brownian motion and $\Lambda$ is a positive scalar.
Note that as $h \to 0$, $\hat X_h \Rightarrow \hat x$ in $C([0,\infty):\R^2)$ and for each $t>0$
$\hat X_h(t)$ has a variance of same order as $X^h(t)$, i.e. $O(h)$.

We now present a search algorithm that makes use of the above heuristic.
\begin{algorithm}
	\label{alg:alg2}
	Suppose at some time instant $n_1$ a sensor has detected a positive number of particles at site 
	$\ti w$.  Fix $r>0$ and $N_0, N_1 \in \N$.  These will be the parameters for the scanning window and time window.
	Also, fix $c \in (0,1)$ and $K \in \NN$.  These parameters will govern the variance and the number of Brownian paths.
	\begin{enumerate}
		\item Compute $\lambda_{n_1}^{N_0}(w)$ as in \eqref{eq:eq958} for all $w$ in the scanning window
		$$\ti \clw_r(\tilde w) = \{w \in \ZZ_h^2: w_2 = \ti w_2 \mbox{ and } |w_1 - \ti w_1| \le r^2h\}.$$
		Let 
		$$w^* = \argmax_{w\in \ti \clw_r(\tilde w)}  \lambda_{n_1}^{N_0}(w), \; \lambda^* =  \max_{w\in \ti \clw_r(\tilde w)} \lambda_{n_1}^{N_0}(w),\; n^* = n_1+N_0.$$
		  Set $L^{(0)} = \sqrt{h}$, $w^{(0)} = w^*$, $\lambda^{(0)} = \lambda^*$, $n^{(0)} = n^*$.
		\item Having defined $(L^{(i)}, w^{(i)}, \lambda^{(i)}, n^{(i)})_{i=0}^j$, define 
		$(L^{(j+1)}, w^{(j+1)}, \lambda^{(j+1)}, n^{(j+1)})$ as follows. 
		Consider $r_j = \lf L^{(j)} r\rf+ 1$ independent standard two dimensional Brownian motions 
		$\{W^{(l)}\}_{l=1}^{r_j}$.  Define
		\begin{equation}
		S_l(w^{(j)}) = \{w^{(j)} - q kh -  L^{(j)} W^{(l)}(kh),\; k = 0, 1, \ldots \lf r/L^{(j)} \rf - 1\}
		\label{eq:eq857}
	\end{equation}
		and 
		$\clw_r^{(j+1)} = \cup_{l=1}^{r_j} S_l(w^{(j)})$.
		\item 
		Compute, for each $w \in \clw_r^{(j+1)}$
			$$\lambda_{n^{(j)}}^{N_1}(w) \doteq \frac{1}{N_1} \sum_{m=n^{(j)}+1}^{n^{(j)}+N_1} N_m(w)$$
			and let
			$$w^{(j+1)} = \argmax_{w\in \clw_r^{(j+1)}}  \lambda_{n^{(j)}}^{N_1}(w), \; \lambda^{(j+1)} =  \max_{w\in \clw_r^{(j+1)}} \lambda_{n^{(j)}}^{N_1}(w),\; n^{(j+1)} = n^{(j)}+N_1.$$
			Finally, let
			\begin{equation}
			L^{(j+1)} = \left \{ 
			\begin{array}{cc}
			 c L^{(j)}  & \text{ if } \lambda^{(j)} \ge \lambda^{(j-1)}+K \\ 
			L^{(j)}/c & \text{ if } \lambda^{(j)} \le \lambda^{(j-1)}-K \\
			L^{(j-1)} & \text{ otherwise } \\
			\end{array} \right.
			\label{eq:eq859}
		\end{equation}
			\item Stop when the sequence $w^{(j)}$ converges.
	\end{enumerate}
	
\end{algorithm}

The main ingredients of Algorithm \ref{alg:alg2} are as follows.  A site $\tilde w$ with positive number of particles is 
likely to be close to the ray $\{e+ qt: t \ge 0\}$.  By placing sensors along the line parallel to $x$-axis that passes through
$\tilde w$ and taking measurements for an initial $N_0$ units of time one can discover a site $w^{(0)}$ which is even closer 
to the ray $\{e+ qt: t \ge 0\}$. Subsequent to this initialization phase the algorithm explores sites in the neighborhood
of $\{w^{(0)}-qt: t \ge 0\}$.  In the first iteration, to determine the location of the $r^2$ sites, $r_0 = \lf L^{(0)} r\rf+ 1$
independent standard two dimensional Brownian motions are used producing the set of sites
$\clw_r^{(1)} = \cup_{l=1}^{r_0} S_l(w^{(0)})$, where
$S_l(w^{(0)})$ is as in \eqref{eq:eq857}.  We then compute, for each $w \in \clw_r^{(1)}$ the average number of particles in
the next $N_1$ units of time and the site where this quantity is maximized becomes the starting point for the search in the next iteration.  In any iteration the variance and the number of the Brownian paths is adjusted according to whether
or not the most recent maximizing site had a significantly larger value of average count in comparison with the maximizing site
from the previous iteration (see \eqref{eq:eq859}).  One can also consider a variation where in the next iteration one backtracks
to a previous maximizing site if the most recent site has too few particles.
As before, convergence  of the algorithm is defined by specifying a suitable tolerance threshold. Section \ref{sec: numerical} will provide some results on the implementation of this algorithm.
\section{Numerical Experiments}\label{sec: numerical}
  
In this section we describe the numerical experiments that were conducted to explore the performance of Algorithms \ref{alg:algstraw} and~\ref{alg:alg2}.  Evolution of the contaminant particle system for various choices of the probability vector $\mathbf{p}$ was simulated. For all these simulations  we took $h = 10\times 2^{-8}$, the source site to be the origin, i.e. $e=(0,0)'$, and the sequence $\{\alpha_n\}$ to be i.i.d. Geometric with mean $h \alpha = 25$.  For numerical purposes the domain $\ZZ^2_h$ needs to be replaced by a bounded box which in our simulations was taken to be 
$\mathcal{B}_s = [-6,6]^2$.  A particle upon reaching the boundary of the box is absorbed. The rationale for making such a modification to the dynamics and for the choice of the bounding box is that, the temporal and spatial scales
at which the evolution of the particle system and our proposed detection algorithms operate, very few particles  reach the boundary of the box by the time the algorithm converges.  Figure~\ref{fig: lambda} 
shows the realization of the random field  $\left( \lambda_{300} (w)\right) _{ w \in \ZZ^2_h \cap \mathcal{B}}$ with $\mathbf{p}=(0.6,0.3,0.025,0.075)$ and with the $x-y$ axis plotted in the units of $h = 10\times 2^{-8}$.

\begin{figure}[h]
\centering
\includegraphics[width=7in, bb=0 0 1917.07 1033]{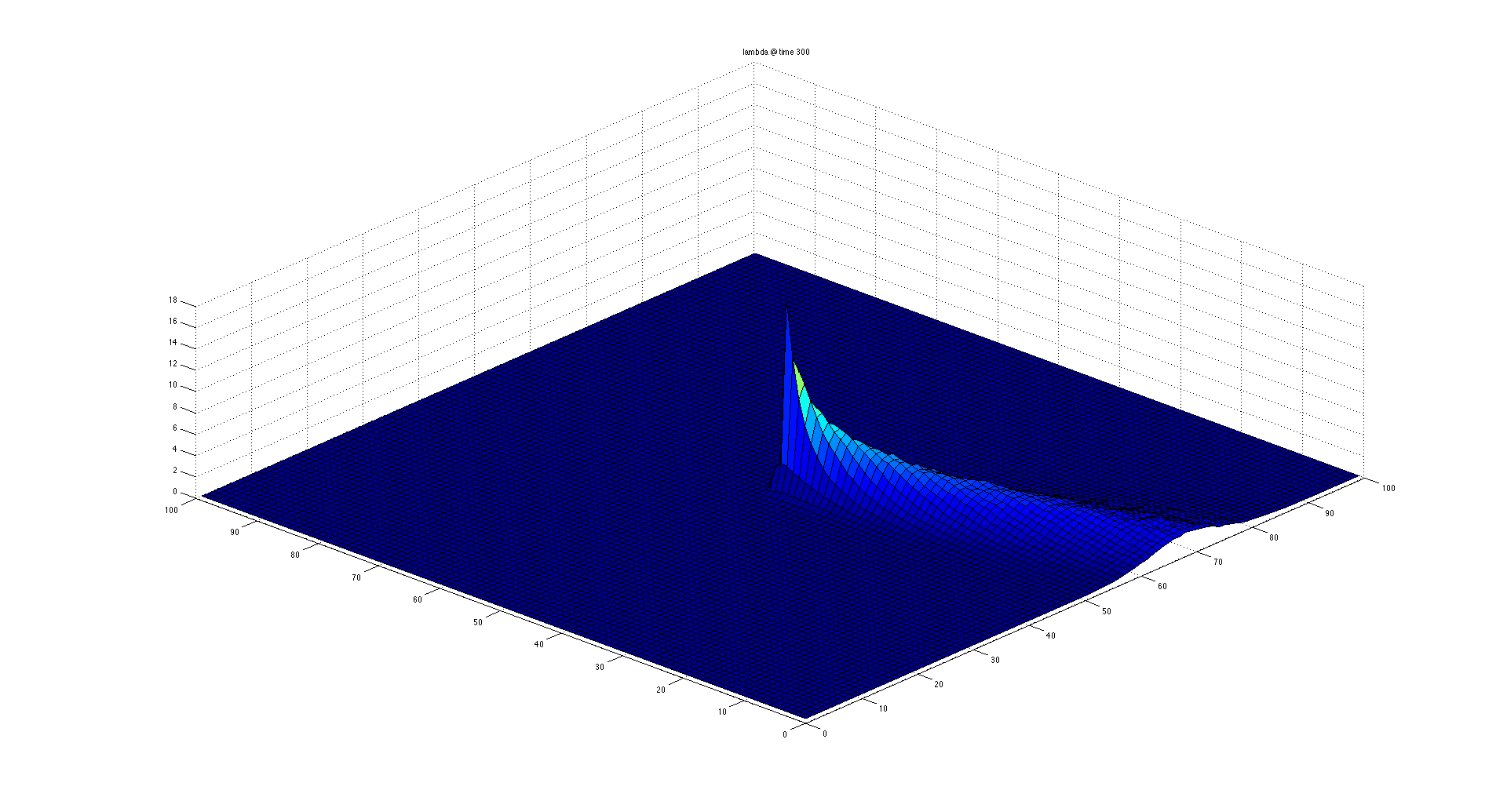}
\caption{Illustration of random field $\left( \lambda_{300} (w)\right) _{ w \in \ZZ^2_h \cap \mathcal{B}}$ with $\mathbf{p}=(0.6,0.3,0.025,0.075)$.}
\label{fig: lambda}
\end{figure}

The figure clearly shows a unique peak for $\lambda_{300}$ at the origin and noting that $q = (0.575, 0.225)'$ one can also see from the figure the approximate hydrodynamic limit behavior of the random field predicted by Theorem \ref{thm:cgce}.  We note that $\lambda_{300}$ corresponds to a time averaging over $300$ time steps whereas Algorithms \ref{alg:algstraw} and~\ref{alg:alg2} will typically use a much shorter time averaging.  To get a sense of how much rougher this random field corresponding to shorter time averaging will be, in Figure~\ref{fig: snapshot} we plot a realization of the random field $N_{75}$ together with its contour curves. This random field realization 
is qualitatively similar to the plot for $\lambda_{300}$ although as expected it is more rough.  Nevertheless one can see from this realization a distinct peak at the origin and also the approximate hydrodynamic limit behavior.
This simulation also illustrates the negligibility of the boundary effect.  In the simulation at the time instant $75$, $1641$ particles were present which lies barely outside the range $\mu \pm \sigma$ where $\mu$ is the expected number of particles in the system at this instant, i.e. $\mu = 75 \times 25= 1875$ and $\sigma$ is the corresponding standard deviation, i.e. $\sigma = h\sqrt{ 75( \alpha-1)\alpha }= 212.13$.  

\begin{figure}[h]
\centering
\includegraphics[width=4.5in, height=5in]{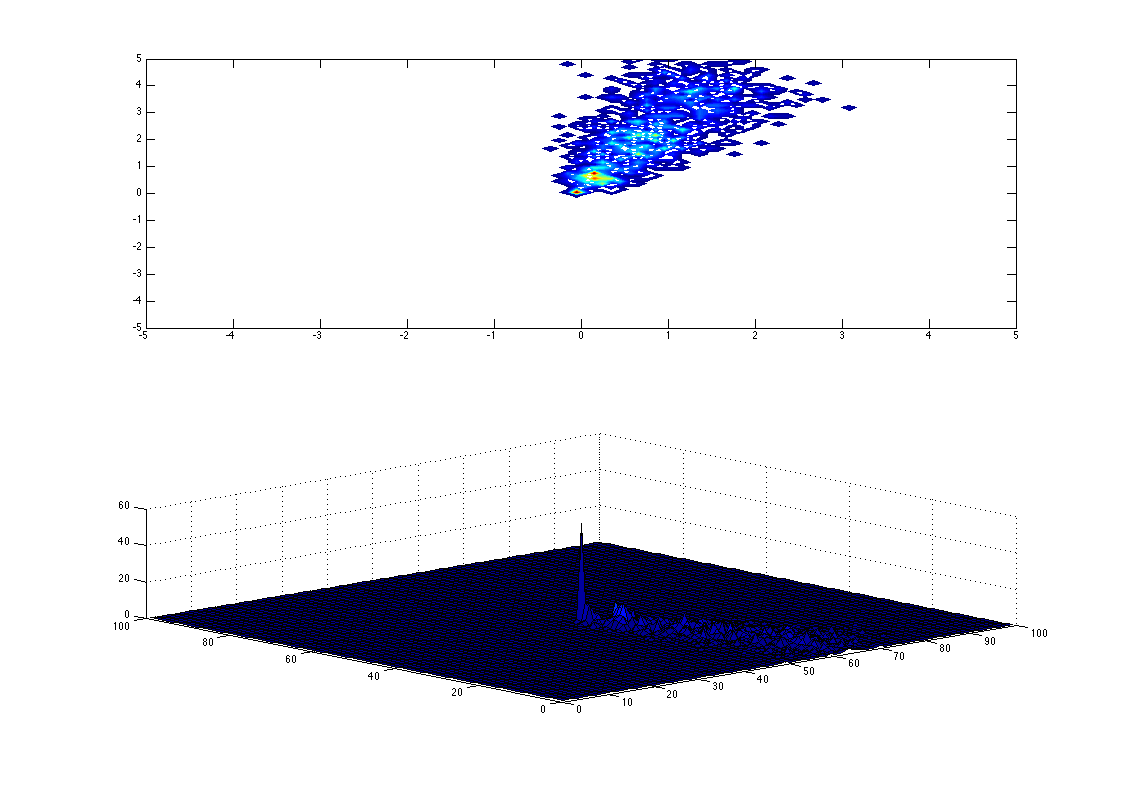}
\caption{Snapshot of $\left( N_{75} (w)\right) _{ w \in \ZZ^2_{0.01} \cap \mathcal{B}}$  with $\mathbf{p}=(0.6,0.3,0.025,0.075)$ both as a contour curve and a surface.}
\label{fig: snapshot}
\end{figure}

We now describe how Algorithms \ref{alg:algstraw} and~\ref{alg:alg2} are implemented.  Both algorithms require an initialization at some site $w_0 \in \ZZ^2_h$.  For this we choose a point uniformly at random among the set of points in the lattice that contain between $1$ and $3$ particles at time instant $n=30$.  In particular both algorithms are initialized in the same manner. Also for both algorithms the convergence of the sequence $w^{(j)}$ is determined in the same way.  We call the detection of the source site successful if the sequence $w^{(j)}$ converges to a site which lies within distance $h$ of the source site, i.e. it belongs to the set $\{(0,0), (0, \pm h), (\pm h, 0)\}$.
Both Algorithms~\ref{alg:algstraw}  and~\ref{alg:alg2} at any time step will use $r^2$ sensors. We experiment with different values of $r$ and study the behavior for different choices of the probability vector $\mathbf{p}$. Each algorithm along with the corresponding random field simulation is implemented
$M=1000$ times for any given choice of the parameters $r$ and $\mathbf{p}$.  To evaluate the performance we compute the relative frequency of the times the algorithm results in successful detection (we refer to this proportion as the `probability of detection' and $1-\mbox{ probability of detection }$ is referred to as the error probability).  We also compute the average number of sensor measurements needed until detection over the $M=1000$ trials, for both algorithms and each set of chosen parameters.  To implement the algorithm we also need to choose the time window parameters $N_0$ and $N_1$.  To arrive at a reasonable choice for these parameters we conduct an experiment with
Algorithm \ref{alg:algstraw} using $r=256$, namely sensor measurements are taken at all the sites in the lattice $[-5,5]^2 \cap \ZZ_h^2$.  One finds that typically with a time window of length $10$ the probability of error is close to $0$.  We show results of one such experiment in Figure~\ref{fig: naivepe}.

\begin{figure}[h]
\centering
\includegraphics[width=4.5in]{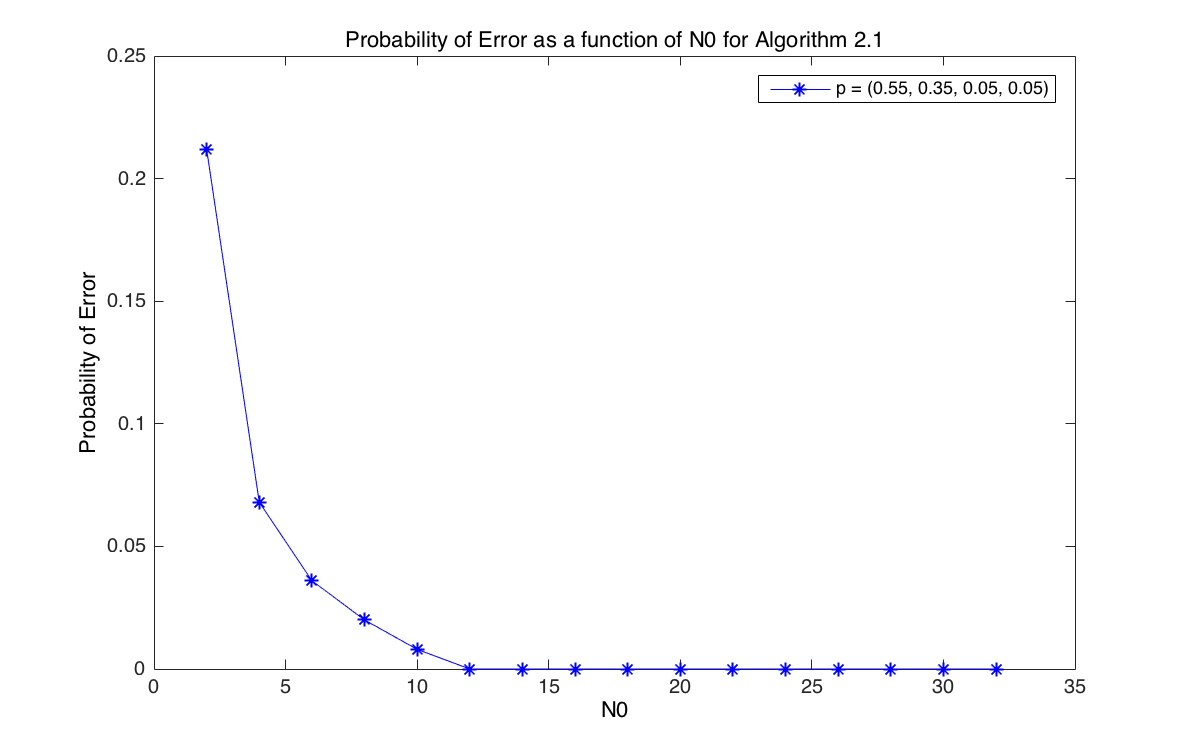}
\caption{Probability of error for Algorithm~\ref{alg:algstraw} as a function of $N_0$ with $r=256$, and $\mathbf{p}=(0.55, 0.35, 0.05, 0.05)$}
\label{fig: naivepe}
\end{figure} 

Guided by these results we take $N_0=N_1 =10$ in all implementations of the two algorithms. The parameters $c$ and $K$ needed for Algorithm \ref{alg:alg2} were taken to be K=0 and c=0.5.
We considered $4$ different choices of the probability vector $\mathbf{p}$:  $\mathbf{p_1}=(0.9,0.05,0.01,0.04), \mathbf{p_2}=(0.70,0.25,0.01,0.04)$, $\mathbf{p_3}=( 0.26,0.26,0.24,0.24 )$, $\mathbf{p_3}=( 0.26,0.26,0.24,0.24 )$, and $\mathbf{p_4}=( 0.55,0.35,0.05,0.05 )$. These cases were chose to cover a range of velocities for the particle motion. We consider two measures for comparison, the average number of sensor measurements needed until the algorithm converges and the probability of error.  We experiment with $r$ ranging from $8$ to $24$ as we find that when $r> 24$, both algorithms rarely fail.
In Figures~\ref{fig: Iterations1} and \ref{fig: Iterations2}  we present the average number of measurements needed for Algorithms~\ref{alg:algstraw}  and \ref{alg:alg2} respectively.
Figures ~\ref{fig: probErr1} and ~\ref{fig: probErr2} give the probability of errors for the two algorithms.

\begin{figure}[h]
\centering
\includegraphics[width=4.5in]{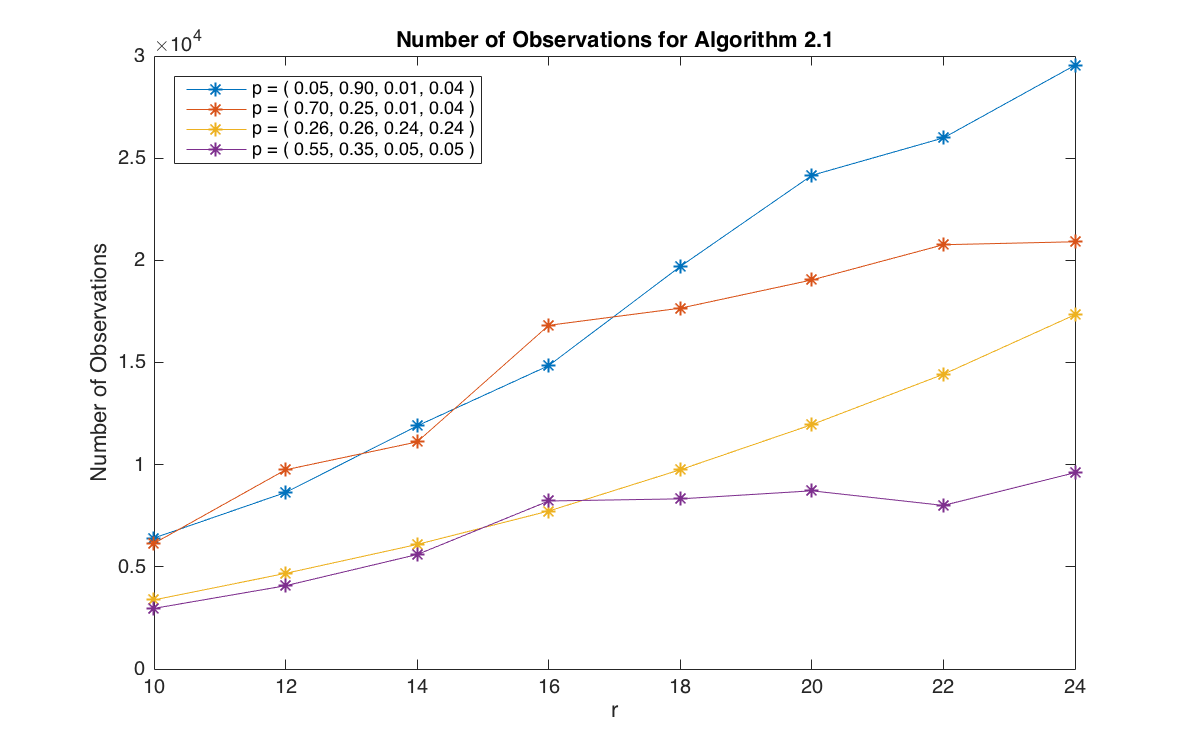}
\caption{Average number of observations as function of $r$ that Algorithm~\ref{alg:algstraw} takes to finish for $\mathbf{p_1}=(0.9,0.05,0.01,0.04), \mathbf{p_2}=(0.70,0.25,0.01,0.04)$, $\mathbf{p_3}=( 0.26,0.26,0.24,0.24 )$, and $\mathbf{p_4}=( 0.55,0.35,0.05,0.05 )$. }
\label{fig: Iterations1}
\end{figure} 
\begin{figure}[h]
\centering
\includegraphics[width=4.5in]{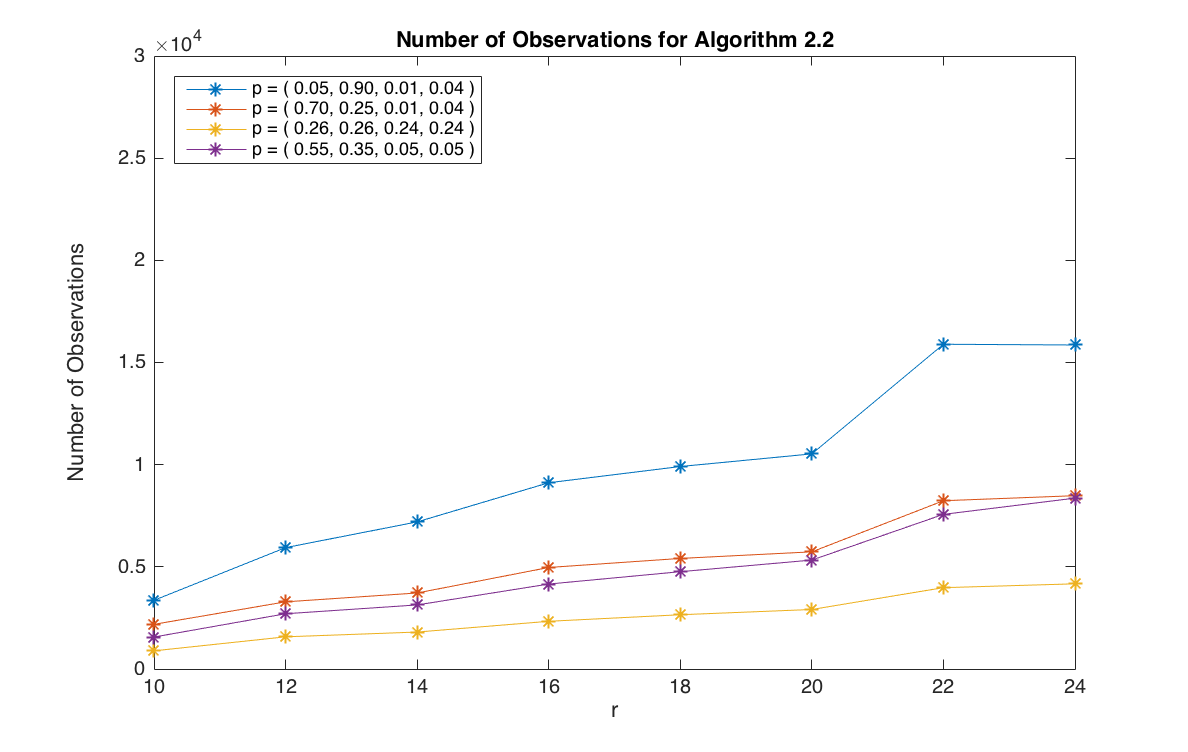}
\caption{Average number of observations as function of $r$ that Algorithm~\ref{alg:alg2} takes to finish for $\mathbf{p_1}=(0.9,0.05,0.01,0.04), \mathbf{p_2}=(0.70,0.25,0.01,0.04)$, $\mathbf{p_3}=( 0.26,0.26,0.24,0.24 )$ and $\mathbf{p_4}=( 0.55,0.35,0.05,0.05 )$, with $c=0.5$  }
\label{fig: Iterations2}
\end{figure}
\begin{figure}[h]
\centering
\includegraphics[width=4.5in]{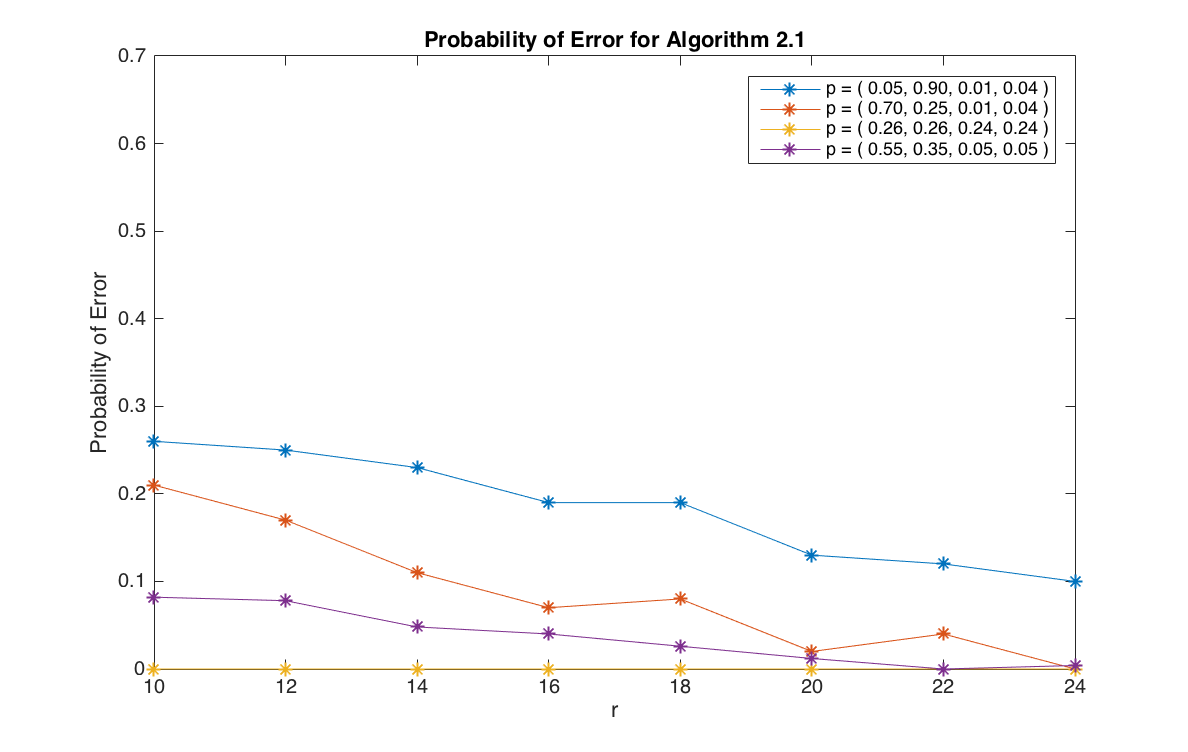}
\caption{Average probability of error as function of $r$ for Algorithm~\ref{alg:algstraw} in cases $\mathbf{p_1}=(0.9,0.05,0.01,0.04), \mathbf{p_2}=(0.70,0.25,0.01,0.04)$, $\mathbf{p_3}=( 0.26,0.26,0.24,0.24 )$ and $\mathbf{p_4}=( 0.55,0.35,0.05,0.05 )$. }
\label{fig: probErr1}
\end{figure} 
\begin{figure}[h]
\centering
\includegraphics[width=4.5in]{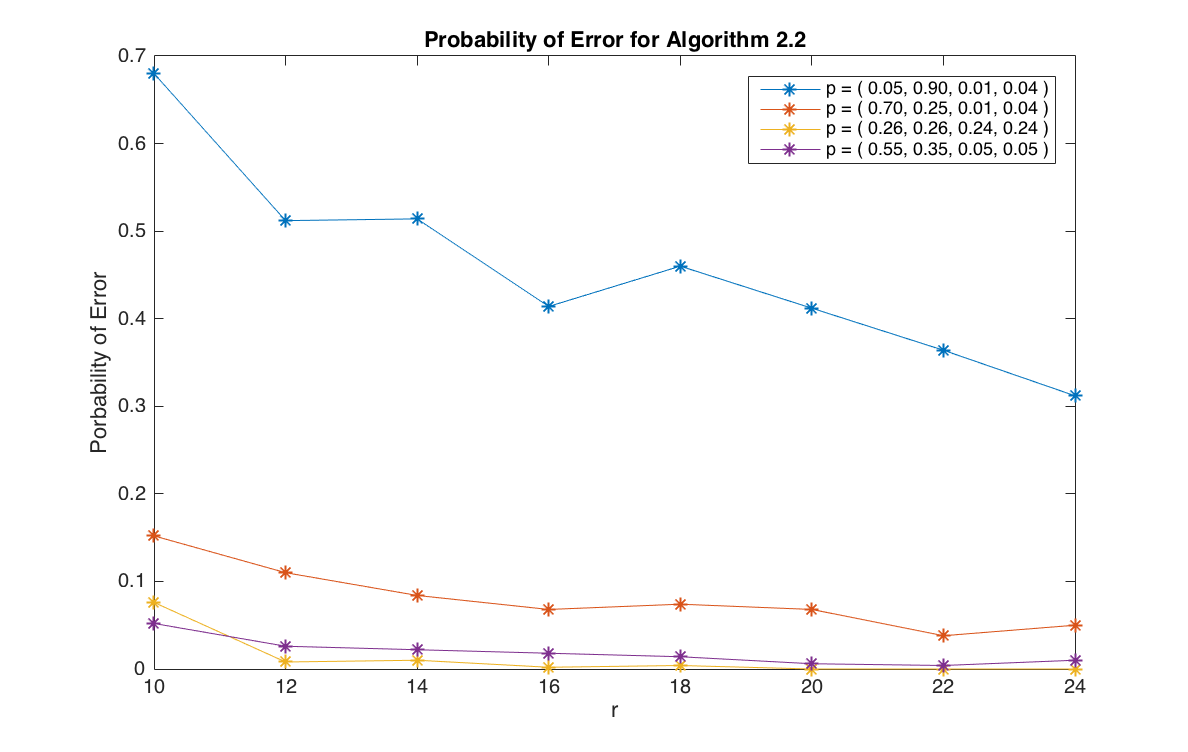}
\caption{Average probability of error as function of $r$ for Algorithm~\ref{alg:alg2} in cases $\mathbf{p_1}=(0.9,0.05,0.01,0.04), \mathbf{p_2}=(0.70,0.25,0.01,0.04)$, $\mathbf{p_3}=( 0.26,0.26,0.24,0.24 )$ and $\mathbf{p_4}=( 0.55,0.35,0.05,0.05 )$ with $c=0.5$ }
\label{fig: probErr2}
\end{figure} 

As expected, as a function of $r$, the probability of error is decreasing while the number of measurements is increasing. We find that the number of sensor measurements needed for Algorithm~\ref{alg:algstraw}  is generally significantly higher (almost twice as many in some instances) than that needed for Algorithm~\ref{alg:alg2}. Also, although in some instances Algorithm~\ref{alg:algstraw} appears to give a lower probability of error than Algorithm~\ref{alg:alg2} for the same value of $r$, we note that in view of the large difference in the number of measurements required for the two algorithms, this comparison is not completely appropriate.
For example, Algorithm~\ref{alg:algstraw}  with $r=18$ and $\mathbf{p} = (0.70,0.25,0.01,0.04)$ gives a probability of error $0.08$ and and requires on average $17,658$ 
measurements whereas Algorithm~\ref{alg:alg2} with $r=24$
and same $\mathbf{p}$ gives a probability of error $0.05$ and requires $8,482$ measurements. In order to make a more accurate comparison we consider the measure
$\mbox{ number of measurements }/ \mbox{ probability of detection }$ which we refer to as the {\em relative efficiency} of the algorithm. We plot the relative efficiency measure for the two algorithms in Figures ~\ref{fig: ri1} and~\ref{fig: ri2} respectively.
These figures clearly demonstrate that Algorithm~\ref{alg:alg2} performs better according to this measure than Algorithm~\ref{alg:algstraw} across all values of $r$ and $\mathbf{p}$. 

\begin{figure}[h]
\centering
\includegraphics[width=4.5in]{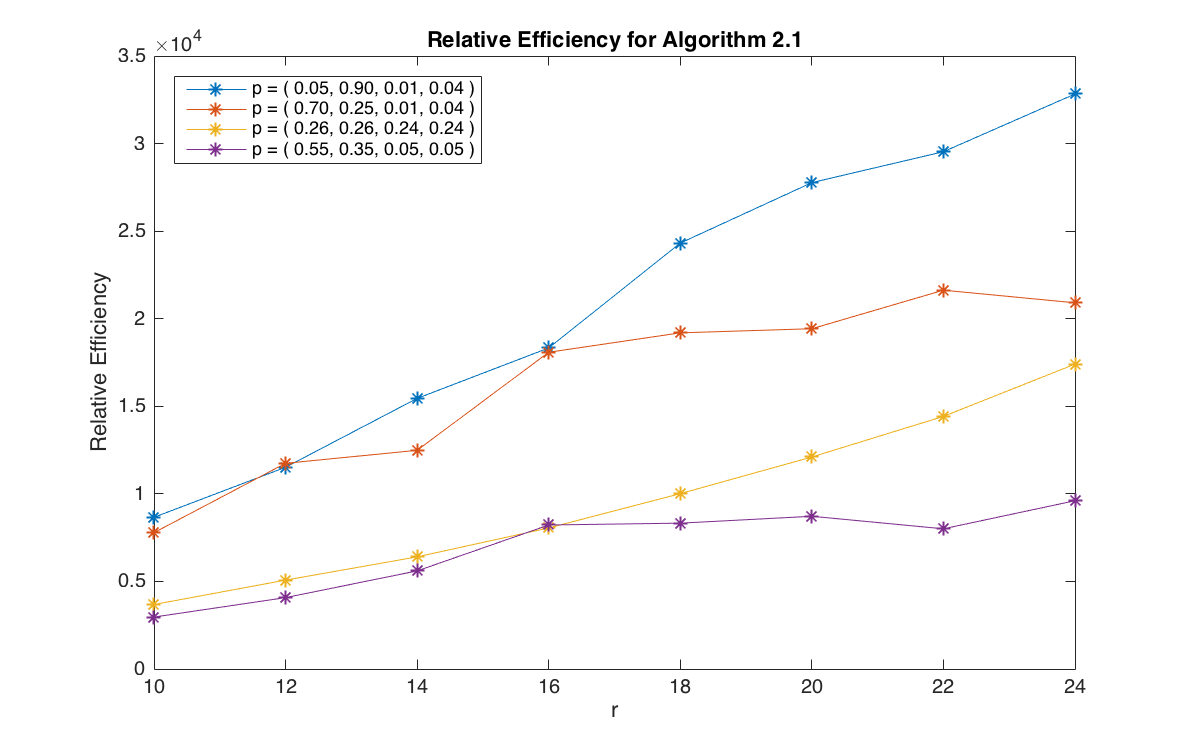}
\caption{Relative Efficiency as function of $r$ for Algorithm~\ref{alg:algstraw} in cases $\mathbf{p_1}=(0.9,0.05,0.01,0.04), \mathbf{p_2}=(0.70,0.25,0.01,0.04)$, $\mathbf{p_3}=( 0.26,0.26,0.24,0.24 )$ and $\mathbf{p_4}=( 0.55,0.35,0.05,0.05 )$. }
\label{fig: ri1}
\end{figure} 

\begin{figure}[h]
\centering
\includegraphics[width=4.5in]{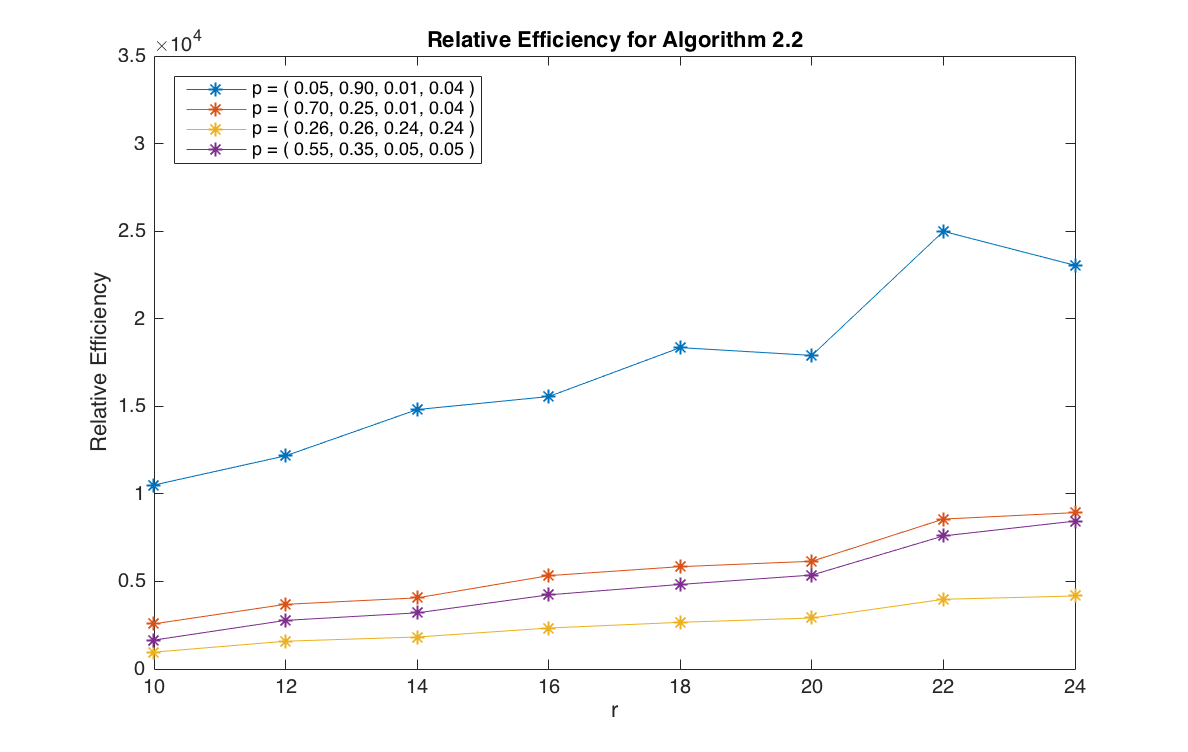}
\caption{Relative Efficiency as function of $r$ for Algorithm~\ref{alg:alg2} in cases $\mathbf{p_1}=(0.9,0.05,0.01,0.04), \mathbf{p_2}=(0.70,0.25,0.01,0.04)$, $\mathbf{p_3}=( 0.26,0.26,0.24,0.24 )$ and $\mathbf{p_4}=( 0.55,0.35,0.05,0.05 )$ with $c=0.5$ }
\label{fig: ri2}
\end{figure} 

In general however we observe that Algorithm~\ref{alg:alg2} does poorly when the probability vector is too degenerate.  By degeneracy we mean here the property that the variance of the two dimensional random variable 
$Z = (Z_1, Z_2)'$ with distribution given by the probability vector $\mathbf{p}$, in the direction normal to the mean of the distribution, is too small.  For example
with $\mathbf{p} = (0.9,0.05,0.01,0.04)$ this variance is $0.09$ while for $\mathbf{p} = ( 0.26,0.26,0.24,0.24 )$ the variance is $0.5$.  The poor behavior of the algorithm when $\mathbf{p}$ is very degenerate can be attributed to  most particles being located in a very small region of the space (namely localized close to the ray $\{qt: t\ge 0\}$). Overall we find that Algorithm \ref{alg:alg2}  is much less sensitive to the values of $r$ and for situations where the probability vector is not too degenerate, there is a considerable improvement obtained by using Algorithm \ref{alg:alg2} over Algorithm \ref{alg:algstraw}.

\bibliographystyle{plain}
\bibliography{sourceDetect}

\begin{thebibliography}{10}

\bibitem{Billingsley}
Patrick Billingsley.
\newblock {\em Convergence of probability measures}.
\newblock Wiley Series in Probability and Statistics: Probability and
  Statistics. John Wiley \& Sons Inc., New York, second edition, 1999.

\bibitem{brennan2004radiation}
Sean~M Brennan, Angela~M Mielke, David~C Torney, and Arthur~B Maccabe.
\newblock Radiation detection with distributed sensor networks.
\newblock {\em Computer}, 37(8):57--59, 2004.

\bibitem{ChMoSo}
Y.~Chen, K.~Moore, and Z.~Song.
\newblock Diffusion boundary determination and zone control via mobile
  actuator-sensor networks (mas-net): challenges and opportunities.
\newblock {\em Proceedings of SPIE: Intelligent Computing: Theory and
  Applications}, 5421:102--113, 2004.

\bibitem{christopoulos2005adaptive}
Vassilios~N Christopoulos and Stergios Roumeliotis.
\newblock Adaptive sensing for instantaneous gas release parameter estimation.
\newblock In {\em Robotics and Automation, 2005. ICRA 2005. Proceedings of the
  2005 IEEE International Conference on}, pages 4450--4456. IEEE, 2005.

\bibitem{Hoeffding}
Wassily Hoeffding.
\newblock Probability inequalities for sums of bounded random variables.
\newblock {\em Journal of the American Statistical Association},
  58(301):13--30, 3 1963.

\bibitem{HHCGPR}
Chunfeng Huang, Tailen Hsing, Noel Cressie, Auroop~R. Ganguly, Vladimir~A.
  Protopopescu, and Nageswara~S. Rao.
\newblock Bayesian source detection and parameter estimation of a plume model
  based on sensor network measurements.
\newblock {\em Appl. Stoch. Models Bus. Ind.}, 26(4):360--361, 2010.

\bibitem{ishida2001plume}
Hiroshi Ishida, Takamichi Nakamoto, Toyosaka Moriizumi, Timo Kikas, and Jiri
  Janata.
\newblock Plume-tracking robots: A new application of chemical sensors.
\newblock {\em The Biological Bulletin}, 200(2):222--226, 2001.

\bibitem{li2001tracking}
Wei Li, Jay~A Farrell, and Ring~T Card.
\newblock Tracking of fluid-advected odor plumes: strategies inspired by insect
  orientation to pheromone.
\newblock {\em Adaptive Behavior}, 9(3-4):143--170, 2001.

\bibitem{nemzek2004distributed}
Robert~J Nemzek, Jared~S Dreicer, David~C Torney, and Tony~T Warnock.
\newblock Distributed sensor networks for detection of mobile radioactive
  sources.
\newblock {\em Nuclear Science, IEEE Transactions on}, 51(4):1693--1700, 2004.

\bibitem{rao2005identification}
N.S.V. Rao.
\newblock Identification of a class of simple product-form plumes using sensor
  networks.
\newblock {\em Innovations and Commercial Applications of Distributed Sensor
  Networks Syposium}, 2005.

\bibitem{sykes1986gaussian}
R.I. Sykes, W.S. Lewellen, and S.F. Parker.
\newblock A {G}aussian plume model of atmospheric dispersion based on
  second-order closure.
\newblock {\em Journal of climate and applied meteorology}, 25(3):322--331,
  1986.

\end{thebibliography}

{\sc
\bigskip
\noi
Sergio A. Almada Monter\\
Department of Statistics and Operations Research\\
University of North Carolina\\
Chapel Hill, NC 27599, USA\\
email: saalm56@gmail.com
\skp

\noi
Amarjit Budhiraja\\
Department of Statistics and Operations Research\\
University of North Carolina\\
Chapel Hill, NC 27599, USA\\
email: budhiraj@email.unc.edu
\skp

\noi
Jan Hannig\\
Department of Statistics and Operations Research\\
University of North Carolina\\
Chapel Hill, NC 27599, USA\\
email: jan.hannig@unc.edu

\skp

}



\end{document}